\documentclass[12pt,oneside,reqno]{amsart}
\usepackage[all]{xy}
\usepackage{amsfonts,amsmath,oldgerm,amssymb,amscd}
\UseComputerModernTips
\numberwithin{equation}{section}
\usepackage[breaklinks]{hyperref}
\allowdisplaybreaks

\usepackage{enumerate}
\usepackage{caption} 
\newtheorem{conjecture}{Conjecture}[section]

\newtheorem{theorem}{Theorem}[section]
\newtheorem{coro}[theorem]{Corollary}

\newtheorem{proposition}[theorem]{Proposition}

\theoremstyle{definition}
\newtheorem{definition}[theorem]{Definition}

\newcommand{\Z}{\mbox{$\mathbb Z$}}
\newcommand{\Q}{\mbox{$\mathbb Q$}}
\newcommand{\R}{\mbox{$\mathbb R$}}     
\newcommand{\C}{\mbox{$\mathbb C$}}     

\begin{document}

	\title[UNIFORM BOUNDS ON $S$-INTEGRAL PREPERIODIC POINTS]{UNIFORM BOUNDS ON $S$-INTEGRAL PREPERIODIC POINTS FOR CHEBYSHEV polynomials} 

\author[R. Padhy]{R. Padhy}
\address{Rudranarayan Padhy, Department of Mathematics, National Institute of Technology Calicut, 
Kozhikode-673 601, India.}
\email{rudranarayan\_p230169ma@nitc.ac.in; padhyrudranarayan1996@gmail.com}

\author[S. S. Rout]{S. S. Rout}
\address{Sudhansu Sekhar Rout, Department of Mathematics, National Institute of Technology Calicut, 
	Kozhikode-673 601, India.}
\email{sudhansu@nitc.ac.in; lbs.sudhansu@gmail.com}

\thanks{
	Work supported by NBHM grant (Sanction Order No: 14053).}

\begin{abstract}
Let $K$ be a number field with algebraic closure $\bar{K}$, let $S$ be a finite set of places of $K$ containing the archimedean places, and  let $\varphi$ be Chebyshev polynomial. In this paper we prove uniformity results on the number of $S$-integral preperiodic points relative to a non-preperiodic point $\beta$, as $\beta$ varies over number fields of bounded  degree. 
\end{abstract}

\maketitle
\pagenumbering{arabic}
\pagestyle{headings}

\section{Introduction}
 Let $K$ be a number field with algebraic closure $\bar{K}$ and let $S$ be a finite set of places of $K$ containing all the archimedean places of $K$ and let $\alpha, \beta \in \bar{K}$. We say that $\beta$ is $S$-integral relative to $\alpha$ if no conjugate of $\beta$ meets any conjugate of $\alpha$ at primes lying outside of $S$. For a rational map $\varphi: \mathbb{P}^1 \to \mathbb{P}^1$ defined over a field $K$, we say that a point $x\in \mathbb{P}^{1}(\bar{K})$ is preperiodic if there exist distinct positive integers $m, n$ such that $\varphi^m(x) = \varphi^n(x)$. We denote by $\mbox{PrePer}(\varphi, \bar{K})$ for the set of preperiodic points of $\varphi$ on $\mathbb{P}^1(\bar{K})$.
 
 In \cite{baker2008}, the following conjecture was made.
 \begin{conjecture}\label{conj1}
 Let $K$ be a number field, and let $S$ be a finite set of places of $K$ that contains all the archimedean places. If $\varphi: \mathbb{P}_K^{1}\to \mathbb{P}_K^{1}$ is a non constant rational function of degree $d\geq 2$ and $\alpha\in \mathbb{P}^{1}(K)$ is non-preperiodic for $\varphi$, then there are at most finitely many preperiodic points $\beta \in \mathbb{P}^{1}(\bar{K})$ that are $S$-integral with respect to $\alpha$.
 \end{conjecture}
 
Conjecture \ref{conj1} has been proven when $\varphi$ is a power map (for $n\geq 2$) map on $\mathbb{G}_m$ and Latt\`es map on elliptic curve  \cite{baker2008}. In \cite{ih}, Ih and Tucker proved that  Conjecture \ref{conj1} is true for Chebyshev polynomials. Petsche \cite{petsche2008} has proved this conjecture for any rational map $\varphi$ of degree $d\geq 2$ under the additional assumptions that $\beta$ is a totally Fatou point. Young \cite{young} studied the quantitative analogue result of Petsche's result. In \cite{yap}, Yap established uniformity result on the number of $S$-integral preperiodic points relative to a non-preperiodic point $\beta$, as $\beta$ varies over number fields of bounded degree for power or Latt\`es map. 
 
The main result of the current paper is to prove a uniform result of \cite[Theorem 1.0.2]{ih}. To state our theorem, we use the following notation. For an algebraic number $x\in \bar{\Q}$ and a number field $K$, we let $G_K(x):=\mbox{Gal}(\bar{K}/K)\cdot x$ denote the $ \mbox{Gal}(\bar{K}/K)$-orbit of $x$ and $|G_K(x)|$ denote the size of the set $G_K(x)$.

  \begin{theorem}\label{thm1}
 Let $K$ be a number field and $S$ be a finite set of places of $\mathbb{Q}$ including all the archimedean places. Suppose $\varphi:\mathbb{P}^1\to \mathbb{P}^1$ is a Chebyshev polynomial and $\beta \in K^{\times}$ is not a preperiodic point of $\varphi$, i.e., $\beta$ is not of the form $\zeta +\zeta^{-1}$ for any root of unity $\zeta$. Then there exists a constant $c=c([K:\mathbb{Q}],S)$ such that for any such $\beta$, if $\alpha$ is preperiodic and $S$-integral relative to $\beta$ then $|G_{\Q}(\alpha)| <c$.
 \end{theorem}
 
 Now from Theorem \ref{thm1}, it is evident to ask that how $c$ is depends on $[K:\Q]$. In our next result, we will provide an upper bound that grows exponentially with $[K:\Q]$. We let $S_{\mbox{fin}}$ denote the subset of $S$ consisting of exactly all non-archimedean places.
 
 \begin{theorem}\label{thm2}
  	Let $K$ a number field and $S$ be a finite set of places of $\mathbb{Q}$ including all the archimedean places. Suppose $\varphi:\mathbb{P}^1\to \mathbb{P}^1$ is a Chebyshev polynomial and $\beta \in K^{\times}$ is not of the form $\zeta +\zeta^{-1}$ for any root of unity $\zeta$. Then there exists a constant $c>0$ such that for all such $\beta$, the set 
	$$\{\alpha \in \mbox{PrePer}(\varphi, \bar{K}): |G_{\Q}(\alpha)|  >c  [K:\Q]^{12}, \alpha\;\text{is $S$-integral relative to $\beta$}\}$$ is a union of at most $|S_{fin}|\;\; \mbox{Gal}(\bar{\Q}/\Q)$-orbits.
  \end{theorem}
For example, if $S = \{2, 3, \infty\}$, then for each $\beta$ with $[\Q(\beta):\Q]\leq D$, except for two possible 
$\mbox{Gal}(\bar{\Q}/\Q)$-orbits as exceptions, all other preperiodic points $\alpha$ that are $S$-integral relative to $\beta$ must satisfy $|G_{\Q}(\alpha)| < cD^{12}$ conjugates, where $c>0$ is a constant independent of $D$ and $\beta$.

\subsection*{Outline of the paper:} At first, we briefly describe Chebyshev dynamical systems and the canonical measure for this system. We give a short introduction related to the notion of height of algebraic numbers and $S$-integrality.  Our proof of main results relies on the (quantitative) equidistribution theorem for dynamical small points on $\mathbb{P}^1$, which is independently proved due to Baker-Rumely \cite{baker2006},  Chambert-Loir \cite{Chambert}, and Favre-Rivera-Litelier \cite{favre}. This result states that if $\{x_n\}$ is a sequence of distinct points in $\mathbb{P}^1(\bar{K})$ such that $h_{\varphi}(x_n)\to 0$, then the sets of $\mbox{Gal}(\bar{K}/K)$-conjugates of the terms $x_n$ equidistribute with respect to the canonical measure $\mu_{\varphi, v}$ on each local analytic space the Berkovich projective line $\mathbb{P}_{\mbox{Berk}, v}^1$ (for details about Berkovich analytification \cite{baker2010, berko}). More precisely, if $\mathcal{P}_n$ is the Galois orbits of $x_n$ for any continuous function $f$ on $\mathbb{P}_{\mbox{Berk}, v}^1$, we have 
\begin{equation}\label{eqdist}
\lim_{n\to \infty} \frac{1}{|\mathcal{P}_n|} \sum_{x\in \mathcal{P}_n} f(x) \to \int fd\mu_{\varphi, v}.
\end{equation}
Let $\beta\in \mathbb{P}^1(K)$. If we have the convergence in \eqref{eqdist} for the function $f(x) = \log^{+}|x-\beta|_v^{-1}$, we say that the sequence $(\mathcal{P}_n)$ satisfies logarithmic equidistribution at $\beta$. Note that if  $(\mathcal{P}_n)$ satisfies logarithmic equidistribution at $\beta$, then only finitely many $x_n$'s can be $S$-integral relative to $\beta$. Then using a result of \cite{yap}, we obtain a bound for the rate of logarithmic equidistribution in terms of the height $h(\beta)$. In Section \ref{sec3}, we prove Theorem \ref{thm2} along with some other results. We note that the main ideas of our work are coherent with those of \cite{yap}.

 \section{Preliminaries}\label{prelim}
 \subsection{The Chebyshev dynamical systems on $\mathbb{P}^1$}\label{cheby}
 We start with a brief description of the Chebyshev dynamical systems on $\mathbb{P}^1$. For more details, (see \cite{milnor} and \cite{sil}).
 
 Define the polynomials \( T_1(z) := z \), \( T_2(z) := z^2 - 2 \), and 
  \[
 T_{n+1}(z) + T_{n-1}(z) = z T_n(z)\quad \mbox{for all}\; n\geq 2.
 \]
 Then a Chebyshev polynomial is defined to be any of the \( T_n \) for \( n \geq 2 \). These polynomials satisfy the following properties \cite{milnor}.
 \begin{enumerate}
 \item For any  $n \geq 1 , T_n(\omega + \omega^{-1}) = \omega^n + \omega^{-n}$, equivalently, $T_n(2 \cos \theta) = 2 \cos(n \theta)$, where  $\omega \in \mathbb{C}^\times$  and  $\theta \in \mathbb{R}$ .
 \item For any  $m, n \geq 1 , T_m \circ T_n = T_{mn}$.
 \item For any  $n \geq 3 ,  T_n$  has $n - 1 $ distinct critical points in the finite plane, but only two critical values, i.e., $\pm 2$.
 \end{enumerate}
\begin{definition}
 Let \( \varphi \) be a Chebyshev polynomial. The dynamical system  induced by $\varphi$ on $\mathbb{P}^1$ (or $\mathbb{A}^1$) is referred to as the (Chebyshev) dynamical system associated with $\varphi$, or the $\varphi$-dynamical system. When \( \varphi \) is clear from the context, we simply call it a Chebyshev dynamical system without explicitly mentioning \( \varphi \).
  \end{definition}
\begin{proposition}\label{prop1}
For any Chebyshev polynomial \( \varphi \), the Julia set of the dynamical system induced by \( \varphi \) (or \( -\varphi \)) is the interval \([-2, 2]\), which can be naturally identified as a subset of the real line in the complex plane.
\end{proposition}
\begin{proof}
See Section 7 of \cite{milnor}.
\end{proof}
 Following proposition will ensure the form of preperiodic points for Chebyshev Dynamical Systems.
\begin{proposition}\label{prop2}
Let \( \varphi \) be a Chebyshev polynomial. Then the finite preperiodic points of the \( \varphi \)-dynamical system are the elements of \( K \) of the form \( \zeta + \zeta^{-1} \), where \( \zeta \) is a root of unity.
\end{proposition}
 \begin{proof}
See Proposition 2.2.2 of \cite{ih}.
\end{proof}
To any dynamical system on $\mathbb{P}^1$ over a number field $K$ are attached some invariant probability measures denoted by $\mu_v$, called the canonical measures ($v$ running over the set of inequivalent primes, finite or infinite, of $K$) on $\mathbb{P}^1(\C_v)$. Let $\varphi$ be a Chebyshev polynomial. The $\varphi$-dynamical system has good reduction everywhere and that the measure $\frac{1}{\pi} \frac{1}{\sqrt{4-x^2}}dx$ on $[-2, 2]$ is a probability measure invariant under $\varphi$ by the recursion formula defining the Chebyshev polynomials, where $dx$ is the usual Lebesgue measure supported on $[-2, 2]$, i.e., if $v\mid \infty$, the canonical measure for the Chebyshev dynamical system is equal to
\begin{equation}
d\mu_v = \frac{1}{\pi} \frac{1}{\sqrt{4-x^2}}dx.
\end{equation}

 \subsection{Heights}\label{height}
  Let $K$ be an algebraic number field  and $\mathcal{O}_{K}$ be its ring of integers. Let $M_{K}$ be the set of places (normalized inequivalent absolute values) of $K$ and we write $K_v$ for the completion of $K$ at the place $v$. For $v\in M_{K}$, if $v$ is an infinite place, then 
\begin{equation}\label{7abc}
|x|_v := |x|^{[K_v:\mathbb{R}]/[K:\mathbb{Q}]}  \quad \hbox{for} \ x \in \mathbb{Q},
\end{equation}
where as if $v$ is a finite place lying above the prime $p$, then 
\begin{equation}\label{7abcd}
|x|_v := |x|_p^{[K_v:\mathbb{Q}_{p}]/[K:\mathbb{Q}]}  \quad \hbox{for} \ x \in \mathbb{Q}.
\end{equation}
For $p \in M_{\mathbb{Q}}$, we choose a normalized absolute value $|\cdot |_p$ in the following way. If $p = \infty$, then $|\cdot |_p$ is the ordinary absolute value on $\mathbb{Q}$, and if $p$ is prime, then the absolute value is the $p$-adic absolute value on $\mathbb{Q}$, with $|p|_p= 1/p$. In either case, we have 
\begin{equation}\label{eq8}
|x|_v:= |N_{K_v/\mathbb{Q}_p}(x)|_p^{1/[K:\mathbb{Q}]},  
\end{equation} 
for $x \in K$ and $v\mid p$. These absolute values satisfy the product formula $\prod_{v\in M_K}|x|_v =1$, for any non zero $x \in K$. Then, more generally, given a finite extension $L/K$, we have
\[\prod_{v\in M_K}\prod_{\sigma}|\sigma(x)|_v =1\] for all $x\in L$, where the inner product is taken over all $[L:K]\; K$-embeddings $\sigma: L\to \C_v$. 

For $x= (x_1:x_2)\in \mathbb{P}^1(\bar{K})$, where $x_1, x_2\in L$, we define the {\em absolute logarithmic height} as 
\begin{equation}\label{eq9}
h(x) = \sum_{v\in M_K}N_v \log \max \{|x_1|_v, |x_2|_v\} =\sum_{v\in M_L} \log \max \{|x_1|_v, |x_2|_v\},
\end{equation}
where $N_v = \frac{[K_v:\Q_v]}{[K:\Q]}$. Note that this definition is independent of the choice of the field $L$ containing $x_1$ and $x_2$, and by the product formula, it is also independent of the choice of projective coordinates for $x$. Furthermore, if $x = x_1/x_2\in \Q$ with $\gcd(x_1, x_2)=1$, then $h(x) = \log \max\{|x_1|, |x_2|\}$ and can be used to bound the maximum number of digits needed to write $x$. Therefore, one may think of the height as measuring the ``arithmetic complexity" of an algebraic number. 

Let $\varphi: \mathbb{P}^1\to \mathbb{P}^1$ be a rational map of degree $d\geq 2$ defined over $K$. For an integer $n\geq 1$, we denote by $\varphi^n$ to mean the $n$-time composition of $\varphi$ with itself. The Call-Silverman canonical height function $h_{\varphi}: \mathbb{P}^1(\bar{K}) \to \mathbb{R}$ relative to $\varphi$ is defined by
\begin{equation}\label{eqcsheight}
h_{\varphi}(x) := \lim_{n\to \infty} \frac{h(\varphi^n(x))}{d^n}
\end{equation}
for all $x\in \mathbb{P}^1(\bar{K})$. It is shown that the limit in \eqref{eqcsheight} exists (\cite{CS}). The canonical height is uniquely characterised by the following properties: for all  $x\in \mathbb{P}^1(\bar{K})$\begin{align*}
&h_{\varphi}\left(\varphi(x)\right) = d h_{\varphi}(x)\quad \mbox{and}\quad |h(x) - h_{\varphi}(x)|< C_{\varphi}
\end{align*}
where $C_{\varphi}$ is an absolute constant. Another basic property of the canonical height is that $h_{\varphi}(x)=0$ for all $x\in \mathbb{P}^1(\bar{K})$ if and only if $x$ is preperiodic (see \cite[Theorem 3.22]{sil}).

\subsection{$S$-integrality}\label{integrality}
Let $S$ be a finite set of places of $K$ containing all the archimedean places, and define the $v$-adic chordal metric on $\mathbb{P}^{1}(\C)$ as 
  \begin{equation}
  \delta_v(x, y) = \frac{|x_1y_2-y_1x_2|_v}{\max\{|x_1|_v, |x_2|_v\}\max\{|y_1|_v, |y_2|_v\}} 
  \end{equation}
where $x=(x_1:x_2)$ and $y= (y_1:y_2)$. Since $0\leq \delta_v(\cdot, \cdot)\leq 1$, we can view $\mathcal{O}_{K,S}$ as the set of points $\gamma\in K$ whose $v$-adic chordal distance to $\infty$ is maximal for all $v\not \in S$; that is, $|\gamma|_v\leq 1$ if and only if $\delta_v(x, \infty) =1$, where $x= (\gamma:1)$.

We set $\lambda_{x, v}(y) = -\log \delta_v(x, y)$ and let $\alpha, \beta \in \mathbb{P}^1(\bar{K})$. Then we say that $\beta$ is $S$-integral relative to $\alpha$ if and only if $\lambda_{x, v}(y) = 0$ for all $v\not \in S$ and for all $x$ and $y$ is varying over the respective $K$-embeddings of $\alpha$ and $\beta$ in $\mathbb{P}^1(\C_v)$. Note that this definition is symmetric. More specifically, since $\delta_v(x, y) = \delta_v(y, x)$, we have that $x$ is $S$-integral relative to $y$ if and only if $y$ is $S$-integral relative to $x$. 

By identifying $x\in \bar{K}$ with $(x:1)\in \mathbb{P}^1(\bar{K})$, we say that $y$ is $S$-integral relative to $x\in \bar{K}$ if and only if, for all $v\not \in S$ and $\sigma, \tau \in \mbox{Gal}(\bar{K}/K)$, 
\begin{align*}
|\sigma(y)-\tau(x)|_v\geq 1,&\quad  \mbox{if}\;\quad |\tau(x)|_v\leq 1,\\
|\sigma(y)|_v\leq 1, &\quad  \mbox{if}\;\quad |\tau(x)|_v>1.
\end{align*}

For two rational maps $\varphi$ and $\psi$ define on $\mathbb{P}^1$ over $K$ each of degree at least two, we denote $<\varphi, \psi>$ the Arakelov-Zhang pairing, which is very closely related to the canonical height functions $h_{\varphi}$ and $h_{\psi}$. In fact, $<\varphi, \psi>$ is defined as a sum of local terms of the form $-\int \lambda_{\varphi, v}\Delta \lambda_{\psi, v}$, where $\lambda_{\varphi, v}$ and  $\lambda_{\psi, v}$ are canonical local height associated to $\varphi$ and $\psi$, where $\Delta$ is $v$adic Laplacian operator. It follows from equidistribution that $\lim_{n\to \infty} h_{\psi}(x_n) = <\psi, \varphi>$, for any infinite sequence of distinct points $\{x_n\}\in \mathbb{P}^1(\bar{K})$ such that $h_{\varphi} (x_n)\to 0$ (see \cite{petsche2012}). For more details about  Arakelov-Zhang pairing (see \cite{ petsche2012, zhang95}). 

For a fixed rational map, we need a bound on the rate of convergence of the Arakelov-Zhang pairing. To state this result, we need the following notations. Let $K$ be a number field, $M_K$ the set of places of $K$ and $\C_v$ the completion of $\bar{\Q}_v$. Given a line bundle $L$ on the projective line $\mathbb{P}^1$ over $K$, an adelic metric $||\cdot||$ on $L$ is a family $||\cdot|| = (||\cdot||_v)$, indexed by the places $v\in M_K$, where each $||\cdot||_v$ is a metric on $L$ over $\C_v$.  The standard height can be described in terms of the standard metric $||\cdot||_{st} = (||\cdot||_{st,v})$ on the line bundle $\mathcal{O}(1)$ over $K$. Given a section $s\in \Gamma (\mathbb{P}^1, \mathcal{O}(1))$ defined over $K$ and a point $x\in \mathbb{P}^1(K)\setminus \{\mbox{div}(s)\}$, we have
\begin{equation}\label{eqheight}
h_{L}(x) = \sum_{v\in M_K}N_v \log ||s(x)||_{st,v}^{-1}
\end{equation}
where $N_v = \frac{[K_v:\Q_v]}{[K:\Q]}$. The canonical height in terms of the canonical adelic metric is defined by	 relative to $\varphi: \mathbb{P}^1\to \mathbb{P}^1$:
\begin{equation}\label{eqheight1}
h_{\varphi}(x) = \sum_{v\in M_K}N_v \log ||s(x)||_{\varphi,\epsilon,v}^{-1}
\end{equation}
for all $x\in \mathbb{P}^1(K)\setminus \{\mbox{div}(s)\}$. Let $\beta\in K$ and consider the section $s(x)= x_0-\beta x_1$ of $\mathcal{O}(1)$. Then we set $\bar{L}_{\beta}$ be the adelic line bundle where for a place $v\in M_K$, we have the metric $\log ||s(x)||_{v, \beta}^{-1} = \lambda_{x, v}(\beta)$. 

For a fixed rational map  $\varphi$ and $\beta\in \mathbb{\bar{K}}$ arbitrary, the following proposition will give the bound on the rate of convergence of the Arkelov-Zhang pairing.

\begin{proposition}\label{prop3}
Let $\varphi$ be a rational map defined on $\mathbb{P}^1(K)$ and let $\bar{L}_{\varphi}$ be a canonical adelization for $\mathcal{O}(1)$. There exists $C_{AZ, \varphi}>0$ depending only on $\varphi$ such that for all $\beta\in \bar{K}$, 
\[|h_{\bar{L}_{\beta}}(\mathcal{P}) - <\bar{L}_{\beta}, \bar{L}_{\varphi}>| \leq h_{\varphi}(\mathcal{P}) +C_{AZ,\varphi}\left(h_{\varphi}(\mathcal{P})^{1/2}+\frac{1+\log |\mathcal{P}|^{1/2}}{|\mathcal{P}|^{1/2}}\right).\]
\begin{proof}
See Proposition 3.6 of \cite{yap}.
\end{proof}

\end{proposition}

\subsection{Linear forms in logarithms}
Here we state a version of Baker’s theorem on lower bounds for linear forms in logarithms. This result is crucial in bounding the distance between algebraic points, (see for more details about this and its application \cite{alan}).

\begin{theorem}\label{linear form logarithms} (Theorem 11.1 of \cite{Yann})
 		Let \(\alpha_1, \alpha_2\) be non-zero complex algebraic numbers and \(\log \alpha_1, \log \alpha_2\) any determinations of their logarithms. Set $D_1 = [\mathbb{Q}(\alpha_1, \alpha_2) : \mathbb{Q}]$. Let \(A_1, A_2\) be real numbers with
\[\log A_j \geq \max \left\{ h(\alpha_j), \frac{\log |\alpha_j|}{D_1}, \frac{1}{D_1} \right\}, \quad j = 1, 2.\]
 Let \(b_1, b_2\) be non-zero integers such that \(b_1 \log \alpha_1 + b_2 \log \alpha_2\) is non-zero and set \[ B = \frac{|b_1|}{D_1 \log A_2} + \frac{|b_2|}{D_1 \log A_1}.\]
 Then we have
 \[\log \left| b_1 \log \alpha_1 + b_2 \log \alpha_2 \right| \geq -21600 D_1^4 (\log A_1)(\log A_2) \left( \max \{ 10, \log B\} \right)^2.\]		
 \end{theorem}
 
 \begin{coro} \label{linearcoro}
 Suppose that $\beta= e^{2 \pi i \theta_0} \in \mathbb{\bar{Q}}$ and $\theta_0\in \R$.Then for any $\epsilon > 0$, there exists a constant $C_\epsilon$ such that for some $a, N\in \Z, (N\neq 0\;\mbox{or}\pm 1)$ with $(a, N) =1$, then either $\frac{a}{N} = \theta_0$ or
 	$$ \log \left | \frac{a}{N}- \theta_0 \right | \geq -C_\epsilon [\mathbb{Q}(\beta):\mathbb{Q}]^3 h \left( \beta \right)N^\epsilon $$ 
 \end{coro}
 \begin{proof}
 Suppose that $\frac{a}{N} \neq \theta_0$. We fix a branch of $\log$ such that $\log \left( e^{2\pi i \theta_0} \right) = 2\pi i \theta_0$ and another branch $\log (1) = 2\pi i$. Then
 		\begin{equation}\label{eq90}
 			\left|\frac{a}{N} - \theta_0 \right| = \frac{1}{2\pi} \left| \frac{a}{N} \cdot 2\pi i - 2\pi i \theta_0 \right| = \frac{1}{2\pi |N|} \left| a \cdot \log(1) - N\log \left( e^{2\pi i \theta_0} \right) \right|.
 		\end{equation}
We will apply Theorem \ref{linear form logarithms} with $\alpha_1 = 1, \, \alpha_2 = \beta, b_1 = a, \ b_2 = -N$ and $D= [\Q(\beta):\Q]$. Since \[\log A_1 \geq \max \left( h(1), \frac{|\log \alpha_1|}{D_1}, \frac{1}{D_1}  \right)\; \mbox{and}\;\;
 \log A_2 \geq \max \left\{ h \left( \beta \right), \frac{|\log \beta|}{D_1}, \frac{1}{D_1} \right\},\]
we set \[\log A_1 = \frac{1}{D_1}, \log A_2 = h(\beta), \; B = \frac{|a|}{D_1 h \left( \beta \right)} + \frac{|N|}{D_1 \log \left( \frac{1}{D_1} \right)}<<|N|. \]
 Since $\log |N|<<|N|^{\epsilon}$, then by Theorem \ref{linear form logarithms}, we have 
 \begin{align}\label{eq91}
 \begin{split}
 \log \left| a \log(1) - N\log \left( e^{2\pi i \theta_0} \right) \right| 
& \geq -21600 D_1^4 \times \frac{1}{D_1} \times h \left( \beta \right) \max \left\{10, \log (C_1N)\right\}^2\\
& \geq -C_{\epsilon}' [\mathbb{Q}(\beta):\mathbb{Q}]^3 h \left( \beta \right)|N|^{2\epsilon}.
\end{split}
 	\end{align}	
Hence, from \eqref{eq90} and \eqref{eq91}, we get the desired inequality.
 \end{proof} 	

\subsection{Quantitative logarithmic equidistribution}
In \cite{ih}, it is shown that for the Chebyshev polynomial $\varphi$ along with any sequence of distinct preperiodic points $(x_n)_{n\ge 1}$, where $x_n=\zeta_n+\zeta_n^{-1}, \zeta_n=e^{\frac{2\pi i}{n}}$ and $\beta\in\mathbb{P}^1(K)$, we have the logarithmic equidistribution at $\beta$. If $h_\varphi(\beta)> 0$, then there are only finitely many preperiodic points that are $S$-integral relative to $\beta$ (\cite[Theorem 3.4]{yap}).  For a fixed place \( v \in M_K \), let \( \mathcal{P}\) be the Galois orbit of some preperiodic point. 

To get a uniform result, we suppose $\mathcal{P}$ is the Galois orbit of some preperiodic point and $v\in M_K$. For $|\mathcal{P}|$ large enough, we will try to find an  upper bound of the form 
\[\left|\frac{1}{\mathcal{P}} \sum_{v\in S}\sum_{x\in \mathcal{P}} N_v\lambda_{x, v} (\beta) - \sum_{v\in S}\int N_v\lambda_{x, v} (\beta)d\mu_{\varphi,v}\right|< \frac{1}{2} h(\beta).\] 
But by the Arakelov-Zhang pairing, we must have
\[\lim_{|\mathcal{P}|\to \infty} \frac{1}{\mathcal{P}} \sum_{v\in S}\sum_{x\in \mathcal{P}} N_v\lambda_{x, v} (\beta) \to h_{\varphi}(\beta) + \sum_{v\in M_K} \int N_v \lambda_{x, v} (\beta)d\mu_{\varphi,v}\] and so if $h_{\varphi}(\beta) > \frac{1}{2} h(\beta)$, which is true for $\beta$ of large height, we obtain that $\mathcal{P}$ can not be $S$-integral relative to $\beta$.

The following result is due to Yap \cite{yap} which gives a bound on the logarithmic equidistribution rate.

\begin{proposition}\label{quntilogprop2}
Let $\varphi$ be the Chebyshev map and $\mathcal{P}$ be the Galois orbit of some preperiodic point. Let $v \in M_K$ be a place of  $K$ and $\beta \in \mathbb{P}^1(K)$ be a point. Fix some $\delta$ with $0<\delta< \frac{1}{2}$ and a rational map $\varphi$ over $K$. Then there exists a constant $C = C(\varphi, \delta) > 0$ such that for any $A > 1$, if 
\[
\max_{x \in \mathcal{P}} \log \left| x - \beta \right|_v^{-1} < A(h(\beta) + 1) |\mathcal{P}|^{1/2 - \delta},
\]
then
\[
\left |\frac{1}{|\mathcal{P}|} \sum_{x \in \mathcal{P}} \lambda_{x, v}( \beta) - \int \lambda_{x, v}( \beta)  \, d\mu_{\varphi,v}\right |
\leq  \frac{C}{|\mathcal{P}|^{\delta}} \sqrt{\log |\mathcal{P}|} A \left(h(\beta) + \log^+ |\beta|_v + 1\right).
		\]
\end{proposition}
\begin{proof}
See Lemma 5.2 of \cite{yap}.
\end{proof}

\section{Proof of Theorem \ref{thm2}}\label{sec3}
Before proving Theorem \ref{thm2}, we need to prove the existence of a constant A in terms a power of the degree $[K(\beta):K]$ for the inequality 
\[\max_{x\in \mathcal{P}} \log|x-\beta|_v^{-1} < A(h(\beta)+1) |\mathcal{P}|^{\frac{1}{2}-\delta}.\]

The following proposition will  provide a bound on the $v$-adic logarithmic distance between a non-preperiodic point $\beta$ and a preperiodic point $x$ in terms of the height of $\beta$ and the size of the  Galois orbit $\mathcal{P}$. Such  bounds  are crucial for controlling the rate of equidistribution in dynamical systems and understanding the distribution of S-integral  points.

\begin{proposition}\label{prop3.1}
Let $\varphi$ be a Chebyshev polynomial and $\mathcal{P}$ be any $\mbox{Gal}(\bar{\mathbb{Q}}/\mathbb{Q})$-orbit of some preperiodic point. Let $K$ be a number field and $\beta\in \mathbb{P}^1(K)$ be a non-preperiodic point and let $v$ be an archimedean place of $K$. For any $\epsilon > 0$, there exists a constant $C_\epsilon$ such that  
	$$ \max_{x\in \mathcal{P}} \log |x-\beta|_v^{-1} < C_\epsilon[K:\mathbb{Q}]^3(h(\beta)+1)|\mathcal{P}|^\epsilon.$$ 
\end{proposition}
\begin{proof}
Given that  $\beta$ is not preperiodic point for the Chebyshev polynomial $\varphi$ and $v\mid \infty$. We may assume that $\beta\not \in [-2, 2]$. By Proposition \ref{prop2}, $\beta \neq \zeta +\zeta^{-1}$ where $\zeta$ is a root of unity. Since $x$ is preperiodic point, then $x$ is of the form $\xi+\xi^{-1}$ where $\xi$ is a root of unity. Now, for a preperiodic point $x$ and  $\sigma\in \mbox{Gal}(K(x)/K)$, we have
\begin{equation*}
|\sigma(x) - \beta|_v \leq |\sigma(x)|_v + |\beta|_v \leq |\beta|_v+2\;\; \end{equation*}
and 
\[|\sigma(x) - \beta|_v \geq ||\sigma(x)|_v - |\beta|_v|_v  \geq |\beta|_v - 2\] since $|\beta|_v>2$. Thus
we obtain $|\beta|_v - 2\leq  |\sigma(x) - \beta|_v \leq |\beta|_v+2.$ That is, the distance to any preperiodic point can be bounded from below by $|\beta|_v - 2$.
Now, $|\beta|_v$ is a real algebraic number in some field $K'$ of degree at most 2 larger than $K$.
Since $|\beta|_v^2 = \beta \bar{\beta} = $ where $\bar{\beta}$ is the complex conjugate, we have 
$ h(|\beta|_v^2) = h(\beta \bar{\beta}) =  h(\beta) + h(\bar{\beta}) = 2h(\beta)$. Also, by height property
\[
h(|\beta|_v - 2) \leq h(|\beta|_v) + h(2) + \log 2.
\]
Therefore,
\[
h(|\beta|_v - 2) \leq h(|\beta|_v) + h(2) + \log 2 \leq 2h(\beta) +2\log 2.
\]
Thus,
\[
\log |\sigma(x)-\beta|_v^{-1}\leq 2(h(\beta) + 1) \leq [K': \Q] (h(\beta) + 1).\] 

Now assume that $\beta \in [-2,2]$. We may write $\beta= e^{2\pi i \theta_0}+ e^{-2\pi i \theta_0} =2\cos(2\pi \theta_0)$ where $\theta_0\in \big(\frac{-1}{2},\frac{1}{2}\big]$. Note that $\beta$ can not be equal to $-2, 2$, or $0$ since we assume that $\beta$ is not preperiodic.  Note that  $x$ is any preperiodic point and is of the form $\xi+\xi^{-1}$ where $\xi$ is a root of unity, and $\sigma\in \mbox{Gal}(K(x)/K)$. We write 
\[\sigma(x) = e^{2\pi i \frac{a}{N}}+ e^{-2\pi i \frac{a}{N}}  = 2 \cos \left(2\pi \frac{a}{N}\right)\]
where $a$ and $N\neq 0$ are integers (depending on $x$) and $\left|\frac{a}{N}\right|\leq 1$. Now for all  $0<\delta<1$ and any conjugate $\sigma(x)$ of $x$ such that $\sigma(x)\in [\beta-\delta, \beta+\delta]$, we have then
\[\log|\sigma(x)- \beta|_v = \log \left|2\cos\left(2\pi \frac{b}{N}\right) -2\cos(2\pi \theta_0)\right|_v\geq \log \left|\frac{b}{N}-\theta_0\right|_v,\]
where $b$ and $N\neq 0$ are integers and $\gcd(b, N) =1$.
By Corollary \ref{linearcoro}, 
\begin{equation}
\log|\sigma(x)- \beta|_v \geq -C_{\epsilon} [K(\beta):K]^3 h(\beta) N^{\epsilon}. 
\end{equation}
Given that the Galois orbit \( \mathcal{P} \) of \( x \) has cardinality \( \varphi(N) \geq \sqrt{N} \), we derive that there exists a constant \( C_\epsilon > 0 \) such that for any \( \epsilon > 0 \), the following inequality holds:
\begin{equation*}
\max_{x \in \mathcal{P}} \log|x- \beta|_v^{-1} \leq C_{\epsilon} [K(\beta):K]^3 (h(\beta) +1)|\mathcal{P}|^{\epsilon}. 
\end{equation*}
\end{proof}
\begin{proposition}\label{prop3.2}
Let $\varphi$ be a Chebyshev polynomial and let $\alpha\in \mbox{PrePer}(\varphi, \bar{K})$. Fix a non-archimedean place $v$ of $\mathbb{Q}$ corresponding to the prime $p$, let $D$ be a positive integer and let $\delta >0$. Then there exists constants $C$ such that for any $\beta \in \mathbb{P}^1(K)$ with $[K:\mathbb{Q}] < D$,  we have 
	$$ \log|\alpha - \beta|_v^{-1} < \delta$$
if $|G_\mathbb{Q}(\alpha)| > C$.
\end{proposition}
\begin{proof}
If \( |\beta|_v > 1 \), then $|\alpha- \beta|_v \leq \max \{ |\alpha|_v,  |\beta|_v \} = |\beta|_v$. Thus, $\log |\alpha - \beta|_v \leq \log |\beta|_v$. Setting  $\delta =\log |\beta|_v>0$, we have 
	\[
	\log |\alpha - \beta|^{-1}_v  < \delta.
	\]
Next suppose $|\beta|_v \leq 1$. Let $\epsilon \in (0,1)$ be a real number. For $\sigma, \tau \in \mbox{Gal}(\bar{K}/K)$, suppose that $ \alpha_k:=\sigma(\alpha) = \zeta_k+\zeta_k^{-1}$ and $\alpha_l:=\tau(\alpha) = \zeta_{\ell}+\zeta_{\ell}^{-1}$ where $\zeta_k$ and $\zeta_{\ell}$ are $k$-th and $\ell$-th root of unity, respectively with $|\alpha_k - \beta|_v < 1-\epsilon,\; |\alpha_{\ell} - \beta|_v < 1-\epsilon$.
Then 
\begin{align*}
1-\epsilon & \geq | (\alpha_{\ell} - \beta) - (\alpha_k - \beta) |_v\\
& = |\alpha_{\ell}-\alpha_k|_v= \left|(\zeta_{\ell}-\zeta_k)- \frac{(\zeta_{\ell}-\zeta_k)}{\zeta_{\ell}\zeta_k}\right|_v\\
&= |\zeta_{\ell}-\zeta_k|_v|1-(\zeta_{\ell}\zeta_k)^{-1}|_v = |1-\zeta_{\ell}^{-1}\zeta_k|_v|1-(\zeta_{\ell}\zeta_k)^{-1}|_v.
\end{align*}
Hence, either $|1-\zeta_{\ell}^{-1}\zeta_k|_v\leq \sqrt{1-\epsilon}$ or $|1-(\zeta_{\ell}\zeta_k)^{-1}|_v\leq \sqrt{1-\epsilon}$. Further, since $p$ is the rational prime under $v$, the only roots of unity $\xi\in \bar{K}_v$ with $|1-\xi|_v\leq 1-\epsilon$ are those with order $p^n$ for some $n\geq 0$ and also $\xi-1$ must be a unit if $\xi$ has order divisible by at least two distinct prime numbers. If $n\geq 1$,
\[1-\epsilon\geq |\xi-1|_v = p^{-\frac{1}{(p-1)p^{n-1}}}\geq p^{-1/p^{n-1}}.\] 
By taking logarithms we will get, $-\log (1-\epsilon) \leq (\log p)/p^{n-1}$. Setting $-\log (1-\epsilon)\geq \epsilon$, we deduce $p^{n-1}\leq (\log p)\epsilon^{-1}$. Let $n_0$ be the largest integer with $p^{n_0}\leq p(\log p)\epsilon^{-1}$. Then $n_0\geq 0$ and $n\leq n_0$, so $\xi^{p^{n_0}} =1$. So we have at most $p(\log p)\epsilon^{-1}$ possibilities of $\xi$ such that $|1-\xi|_v\leq 1-\epsilon$. Therefore, there are only $C:= p(\log p)\epsilon^{-1}$ many Galois conjugates of $\alpha$ such that $|\alpha- \beta|_v<p^{-\frac{1}{(p-1)p^{n-1}}}$. Thus, for $|G_K (\alpha)|> C$, we have $|\alpha-\beta|_v \geq p^{-\frac{1}{(p-1)p^{n-1}}}$ and hence
\[\log|\alpha-\beta|_v^{-1}<\frac{1}{p^{n-1}(p-1)}\log p <\delta.\]
\end{proof}
From the proof of the above result, we have the following:
\begin{coro}\label{cor3.3}
	Let  $\varphi$ be the Chebyshev polynomial and $v$ of $\mathbb{Q}$ be a non-archimedean place. Then for any $\beta \in \mathbb{P}^{1}(\bar{\mathbb{Q}})$, there do not exist two distinct preperiodic points $\alpha_1,\alpha_2$ such that $$\log|\beta-\alpha_i|_v^{-1} \ge\frac{1}{p-1}\log p,
	$$ where $p$ is a prime number corresponding to the place $v$
\end{coro}
\begin{proof}
Let $\beta \in \mathbb{P}^{1}(\bar{\mathbb{Q}})$ be arbitrary. Suppose on contrary, we have
		\begin{align*}
		&\log|\beta-\alpha_i|_v^{-1} \ge \frac{1}{p-1}\log p = \log p^{\frac{1}{p-1}}
		\end{align*} for $i=1, 2$.
This implies $|\beta-\alpha_i|_v < 1/p^{\frac{1}{p-1}}$.  Thus, 
\[|\alpha_1-\alpha_2|_v =|(\alpha_1-\beta) - (\alpha_2-\beta)|_v= |1-\zeta_m^{-1}\zeta_n \big|_v|1-\big(\zeta_m\zeta_n\big)^{-1}\big|_v <1/p^{\frac{1}{p-1}},\] which is a contradiction, as $|1-\zeta|_v$ is at least $1/p^{\frac{1}{p-1}}$ for any root of unit $\zeta$.
\end{proof}

\subsection{Proof of Theorem \ref{thm2}}
First we consider a finite extension $L$ of $K$ of degree at most $D$ such that for each non-archimedean place $v\in S$ Chebyshev dynamical system has good reduction at $v$. Indeed, for any finite extension $L$ of $K$, we can take $S_L$ to be the set of primes in $L$ lying over the primes in $S$ and note that the points $S_L$-integral relative to $\beta$ contains those points which are $S$-integral relative to $\beta$. Therefore, proving the theorem for the larger field $L$ establishes it for the smaller field $K$. 

Let $\epsilon>0$ and $\beta\in K^{\times}$ be non-preperiodic (i.e., not of the form $\zeta+\zeta^{-1}$ for any root of unity $\zeta$). Applying Proposition \ref{prop3.1} and  Corollary \ref{cor3.3}, we can find a constant $A_{\epsilon}$ such that for place $w\in S_L$, we have
\begin{equation*}
 \max_{\alpha\in \mathcal{P}} \log |\alpha-\beta|_w^{-1} < A_\epsilon D^3(h(\beta)+1)|\mathcal{P}|^\epsilon.
\end{equation*}
for all $\mbox{Gal}(\bar{K}/K)$-orbits $\mathcal{P}$ of preperiodic points with the possible exception of one orbit for each place $w$. But if both $w$ and $r$ are above the same place $v$, they differ by an automorphism of $\mbox{Gal}(\bar{K}/K)$, and so we will pick up the same $\mbox{Gal}(\bar{K}/K)$-orbit $\mathcal{P}$ as an exception. Thus, we get total $|S_{\mbox{fin}}|$ exceptions.

 Now let $\alpha$ be a preperiodic point and is not one of the  above exceptions. Then by Proposition \ref{quntilogprop2}, there is a constant $C>0$, depending only on $\varphi$, such that for any $w\in S_L$, we have
	\begin{equation}\label{eq3.2}
		\Big|\frac{1}{|\mathcal{P}|}\sum_{\alpha\in \mathcal{P}} \lambda_{\alpha, w}(\beta)-\int \lambda_{\alpha, w}(\beta) d\mu_{\varphi, w}\Big| \le \frac{C}{|\mathcal{P}|^\delta} \sqrt{\log |\mathcal{P}|}A(h(\beta) + \log^+|\beta|_w+1)
	\end{equation}  
Setting $\delta=\frac{1}{2}-\epsilon, A=A_\epsilon (\log D)^2D^3$ in \eqref{eq3.2}, we get
	\begin{align*}
		\Big|\frac{1}{|\mathcal{P}|}\sum_{\alpha\in \mathcal{P}} &\lambda_{\alpha, w}(\beta)-\int \lambda_{\alpha, w}(\beta) d\mu_{\varphi, w}\Big| \\
		&\le \frac{CA_\epsilon}{|\mathcal{P}|^{\frac{1}{2}-\epsilon}} \sqrt{\log |\mathcal{P}|}D^3(\log D)^2(h(\beta) + \log^+|\beta|_w+1).
\end{align*}
For any constant $N > 0$ assuming that $|\mathcal{P}|>C|S_L|^3D^{12}$ for some suitable $C$ gives, 
	\begin{equation*}
		\Big|\frac{1}{|\mathcal{P}|}\sum_{\alpha\in \mathcal{P}}\lambda_{\alpha, w}(\beta)-\int \lambda_{\alpha, w}(\beta) d\mu_{\varphi, w}\Big| \le \frac{h(\beta) + \log^+|\beta|_w+1}{N|S_L|D^{2.5}}.
	\end{equation*}
Summing up over all places in $S_L$, as $\sum_{w \in S_L} N_w \log^+|\beta|_w \le h(\beta)$, by increasing our constant $C$ we get	\begin{align*}
			\Big|\frac{1}{|\mathcal{P}|}\sum_{w \in S_L}\sum_{\alpha\in \mathcal{P}}N_w \lambda_{\alpha, w}(\beta)&-\sum_{w \in S_L} \int N_w \lambda_{\alpha, w}(\beta) d\mu_{\varphi, w}\Big| \\
			&\le \sum_{w \in S_L} N_w \frac{h(\beta) + \log^+|\beta|_w+1}{N|S_L|D^{2.5}} \\
			& = \sum_{w \in S_L} N_w \frac{h(\beta)+1}{N|S_L|D^{2.5}} + \sum_{w \in S_L} N_w \frac{\log^+|\beta|_w}{N|S_L|D^{2.5}} \\
			& \le \frac{h(\beta)+1}{ND^{2.5}} + \frac{h(\beta)}{N|S_L|D^{2.5}} = \frac{h(\beta) +1}{ND^{2.5}}.
\end{align*}
The last inequality holds for sufficiently large of $|S_L|$. 
Since $\mathcal{P}$ is any Galois orbit of some preperiodic point of  the Chebyshev map $\varphi$, then $h_\varphi(\mathcal{P})=0$. By using Proposition \ref{prop3}, we will get 
	\begin{align*}
		\big| h_{\bar{L}_{\beta}}(\mathcal{P})-\big<\bar{L}_\beta,\bar{L}_\varphi\big>\big| &\le C_{AZ, \varphi} \Big( \frac{1+\log |\mathcal{P}|^{1/2}}{|\mathcal{P}|^{1/2}}\Big) <\frac{1}{D^{2.5}}.
	\end{align*}
   As $\mathcal{P}$ is $S$-integral relative to $\beta$, we have $\lambda_{\alpha, w}(\beta)=0$ for all $w \notin S_L$ and $\alpha \in \mathcal{P}$.
	Therefore,
	\begin{align*}
			\frac{1}{|\mathcal{P}|}\sum_{w \in M_{K}}\sum_{\alpha\in \mathcal{P}}N_w \lambda_{\alpha, w}(\beta)-\sum_{w \in M_{K}} \int N_w\lambda_{\alpha, w}(\beta) d\mu_{\varphi, w} &  =h_{\bar{L}_\beta}(\mathcal{P})-\big<\bar{L}_\beta,\bar{L}_\varphi\big> +h_\varphi(\beta)\\
			& \ge \frac{-1}{D^{2.5}}+h_{\varphi}(\beta).
	\end{align*}
So, we have 
\begin{align*}
h_\varphi(\beta) = h(\beta)&\le \frac{h(\beta)+1}{ND^{2.5}} + \frac{1}{D^{2.5}}.
	\end{align*}
If $h(\beta) \ge 1$, then $ND^{2.5}\leq N+2$, which is not possible for $D>1$. Now if  $h(\beta) <1$, then by using Dobrowolski's result \cite{dob} implies that 
	$$h(\beta) \ge \frac{C}{D(\log D)^3}.$$
This gives,
	\begin{align*}
		\frac{C}{D(\log D)^3}\Big(1-\frac{1}{ND^{2.5}}\Big)
		 &\le \frac{1}{ND^{2.5}}+\frac{1}{D^{2.5}}.
	\end{align*}
If $N$ is large enough, then $\frac{C}{D(\log D)^3} \le \frac{1}{D^{2.5}}$ which is a contradiction. Hence, $|\mathcal{P}| \le C|S_L| D^{12}$. This completes the proof of Theorem \ref{thm2}. \qed

Similarly, we can deduce the proof of Theorem \ref{thm1} from Theorem \ref{thm2} by using Proposition \ref{prop3.2} instead of Corollary \ref{cor3.3}.


\begin{thebibliography}{99}
\bibitem{alan} A. Baker, \emph{Transcendental Number Theory}, Cambridge University Press,  New York, Cambridge, 1975.

\bibitem{baker2008} M.Baker, S. Ih and R. Rumley, A finiteness property of torsion points, {\em Algebra Number Theory} \textbf{2} (2008) 217--248.

\bibitem{baker2006} M. Baker and R. Rumely, Equidistribution of small points, rational dynamics, and potential theory, {\em Ann. Inst. Fourier (Grenoble)} {\bf 56} (2006) 625–688.

\bibitem{baker2010} M.  Baker and R. Rumely, {\em Potential theory and dynamics on
the Berkovich projective line} {\em Mathematical Surveys and Monographs}, vol 159, Amer. Math. Soc., Providence, 159 (2010).

\bibitem{berko} V. Berkovich, {\em Spectral theory and analytic geometry over non-Archimedean fields}, AMS Mathematical Surveys and Monographs, vol 33, Amer. Math. Soc., Providence, RI, 1990.

\bibitem{Yann}Y. Bugeaud, \emph{Linear Forms in Logarithms and Applications}, European Mathematical Society, Zurich, 2018.

\bibitem{CS} G. S. Call and J. H. Silverman, Canonical heights on varieties with morphisms, {\em Compos. Math.} {\bf  89} (1993) 163- 205.

\bibitem{Chambert}A. Chambert-Loir, Mesures et \'equidistribution sur les espaces de Berkovich, {\em J. Reine Angew. Math. } {\bf 595} (2006) 215-235.

\bibitem{dob} E. Dobrowolski, On a question of Lehmer and the number of irreducible
factors of a polynomial, {\em Acta Arith.} {\bf  34} (1979) 391-401.


\bibitem{favre} C. Favre and J. Rivera-Letelier, \'Equidistribution quantitative des points de petite hauteur sur la droite projective, {\em Math. Ann.}  {\bf 335}(2) (2006) 311-361.

\bibitem{grant}D. Grant and S. Ih, Integral division points on curves, {\em Compos. Math.} {\bf149}(12) (2013) 2011-2035.

\bibitem{ih}S. Ih and T. Tucker, A finiteness property for preperiodic points of Chebyshev polynomials, {\em Int. J. Number Theory} {\bf 6}(5) (2010) 1011-1025. 

\bibitem{milnor} J. Milnor, \emph{Dynamics in One Complex Variable}, Ann. of Math. Stud., vol. 180, Princeton University Press, Princeton-Oxford, 2006.

\bibitem{neukrich} J. Neukirch, \emph{Algebraic Number Theory}, Vol. 322, Springer Science \& Business Media, Chicago, Harvard, 2013.

\bibitem{petsche2008} C. Petsche, $S$-integral preperiodic points by dynamical systems over number fields, {\em Bull. Lond. Math. Soc.}  {\bf 40} (5) (2008) 749-758.

\bibitem{petsche2012} C. Petsche, L. Szpiro, and T. J. Tucker, A dynamical
pairing between two rational maps, {\em Trans. Amer. Math. Soc.} {\bf 364}(4) (2012) 1687-1710. 

\bibitem{sil} J. H. Silverman, \emph{The Arithmetic of Dynamical Systems}, Springer,  New York, 2007.

\bibitem{yap} J. W. Yap, Uniform bounds on $S$-integral preperiodic points for power and latt\`es maps, 
(2023) arXiv:2302.01562v2.

\bibitem{young} M. Young, Effective bounds on S-integral preperiodic points for polynomials, 
(2022) arXiv:2206.14252.


\bibitem{zhang95} S. Zhang, Small points and adelic metrics, {\em J. Alg. Geom.} {\bf 4}(2) (1995) 281-300.
\end{thebibliography}
\end{document}